\documentclass[12pt]{amsart}
\usepackage[margin=1.15in]{geometry}
\usepackage{amscd, amssymb, amsmath, wasysym}
\usepackage{graphicx}
\usepackage{amsfonts}
\usepackage{mathrsfs}    
\usepackage{eucal}     
\usepackage{latexsym}   
\usepackage{verbatim}   
\usepackage[all]{xy}   
\usepackage[dvipsnames]{xcolor}
\usepackage{bookmark}

 \usepackage{hyperref}
 \hypersetup{
     colorlinks=true,
     linkcolor=NavyBlue,
     filecolor=NavyBlue,
     citecolor = TealBlue,
     urlcolor=magenta,
  }

\theoremstyle{definition}

\theoremstyle{remark}

\theoremstyle{corollary}

\theoremstyle{theorem}

\theoremstyle{corollary}

\newtheorem{theorem}{Theorem}[section]
\newtheorem{lemma}[theorem]{Lemma}

\theoremstyle{corollary}
\newtheorem{corollary}[theorem]{Corollary}

\theoremstyle{definition}

\newtheorem{example}[theorem]{Example}

\theoremstyle{remark}
\newtheorem{remark}[theorem]{Remark}

\numberwithin{equation}{section}

\newcommand{\kk}{\mathbf{k}}

\def\nP{\mathbf{P}}

\def\codim{\operatorname{codim}}

\def\Sym{\operatorname{Sym}}

\DeclareMathOperator{\ND}{ND}

\newcommand{\suchthat}{\;\ifnum\currentgrouptype=16 \middle\fi|\;}

\input xy
\xyoption{all}

\title[Upper bounds for graded Betti numbers]{Upper bounds for graded Betti numbers of projective schemes in the first nontrivial strand}

\begin{document}

\newcommand{\kaistaddress}{Department of Mathematical Sciences, KAIST, 291 Daehak-ro, Yuseond-gu, Daejeon 34141, Republic of Korea}

\author{Doyoon Ha}
\address{\kaistaddress}
\curraddr{School of Mathematics, Georgia Institute of Technology, 686 Cherry Street, Atlanta, Georgia 30332, USA}
\email{dha49@gatech.edu}

\author{Minjae Kwon}
\address{\kaistaddress}
\email{tngkrwkalswo@kaist.ac.kr}

\author{JeongDon Lee}
\address{\kaistaddress}
\email{s1865915@kaist.ac.kr}

\author{Jinhyung Park}
\address{\kaistaddress}
\email{parkjh13@kaist.ac.kr}

\date{\today}
\subjclass[2020]{14N05, 13D02}
\keywords{the geometry of syzygies, graded Betti number, Boij--S\"{o}derberg theory}

\thanks{J.P.~was partially supported by the National Research Foundation (NRF) funded by the Korea government (MSIT) (RS-2026-25478877).}


\begin{abstract}
Using Boij--S\"{o}derberg theory, we give a quick new approach to the results of Han--Kwak \cite{HK} and Ahn--Han--Kwak \cite{AHK} on upper bounds for graded Betti numbers of projective schemes in the first nontrivial strand. 
\end{abstract}

\maketitle

\section{Introduction}

We work over an algebraically closed field $\kk$ of arbitrary characteristic. Let $V$ be a finite dimensional vector space over $\kk$, and $M$ be a finitely generated graded module over $S:=\Sym V$. By the Hilbert syzygy theorem, $M$ admits a graded minimal free resolution
$$
0 \longleftarrow M \longleftarrow E_0 \longleftarrow E_1 \longleftarrow \cdots \longleftarrow E_p \longleftarrow E_{p+1} \longleftarrow \cdots,
$$
where
$$
E_p = \bigoplus_q K_{p,q}(M, V) \otimes S(-p-q).
$$
We may regard $K_{p,q}(M, V)$ as the space of $p$-th syzygies of weight $q$ of $M$. Note that  $K_{p,q}(M, V)$ can be computed as the cohomology of the Koszul-type complex
$$
\wedge^{p+1} V \otimes M_{q-1} \longrightarrow \wedge^p V \otimes M_q \longrightarrow \wedge^{p-1} V \otimes M_{q+1}.
$$
The graded Betti numbers
$$
\kappa_{p,q}=\kappa_{p,q}(M, V):= \dim K_{p,q}(M,V)
$$
form the Betti table
\begin{center}
\texttt{ \begin{tabular}{c|cccccc}
         & $0$ & $1$ & $\cdots$ & $p$ & $p+1$ & $\cdots$ \\ \hline
    $0$    & $\kappa_{0,0}$ & $\kappa_{1,0}$ & $\cdots$ & $\kappa_{p,0}$ & $\kappa_{p+1,0}$ & $\cdots$ \\
    $1$   & $\kappa_{0,1}$ & $\kappa_{1,1}$ & $\cdots$ & $\kappa_{p,1}$ & $\kappa_{p+1,1}$ & $\cdots$ \\
    $\vdots$ & $\vdots$ &  $\vdots$ &  $\ddots$ &  $\vdots$ &  $\vdots$ &  $\ddots$ \\
    $q-1$ & $\kappa_{0,q-1}$ & $\kappa_{1,q-1}$ & $\cdots$ & $\kappa_{p,q-1}$ & $\kappa_{p+1,q-1}$ & $\cdots$ \\
    $q$ & $\kappa_{0,q}$ & $\kappa_{1,q}$ & $\cdots$ & $\kappa_{p,q}$ & $\kappa_{p+1,q}$ & $\cdots$ \\
    $\vdots$ & $\vdots$ &  $\vdots$ &  $\ddots$ &  $\vdots$ &  $\vdots$ &  $\ddots$ \\
\end{tabular}}
\end{center}
The width of the Betti table is the projective dimension of $M$, which is $r+1-\operatorname{depth}(M)$ by the Auslander--Buchsbaum formula, and the height is the Castelnuovo--Mumford regularity of $M$. We refer to \cite{AN} and \cite{EL} for more details on the geometry of syzygies.

\medskip

We are mainly interested in the case that $M$ is the homogeneous coordinate ring $S(X):=S/I_{X/\nP^r}$ of a non-degenerate projective scheme $X \subseteq \nP V = \nP^r$, where $\lvert V \rvert$ is a very ample linear system on $X$. It is easy to see that $\kappa_{p,0}(S(X),V)=0$ for $p \geq 1$ and $\kappa_{0,q}(S(X), V)=0$ for $q \geq 1$ but $\kappa_{0,0}(S(X),V)=1$. Notice that
\begin{equation}\label{eq:h^0(I(q))}
h^0(\nP^r, \mathscr{I}_{X/\nP^r}(q))=0
\end{equation}
if and only if  the Betti table of $S(X)$ is of the following form:
\begin{center}
\texttt{ \begin{tabular}{c|ccccccc}
         & $0$ & $1$ & $2$ & $\cdots$ & $p$ & $p+1$ & $\cdots$ \\ \hline
    $0$    & $1$ & $-$ & $-$ & $\cdots$ & $-$ & $-$ & $\cdots$ \\
    $1$   & $-$ & $-$ & $-$ & $\cdots$ & $-$ & $-$ & $\cdots$ \\
    $\vdots$ & $\vdots$ & $\vdots$ & $\vdots$ &  $\ddots$ &  $\vdots$ &  $\vdots$ &  $\ddots$ \\
    $q-1$   & $-$ & $-$ & $-$ & $\cdots$ & $-$ & $-$ & $\cdots$ \\
    $q$ & $-$ & $\kappa_{1,q}$ & $\kappa_{2,q}$ & $\cdots$ & $\kappa_{p,q}$ & $\kappa_{p+1,q}$ & $\cdots$ \\
    $\vdots$ & $\vdots$ & $\vdots$ & $\vdots$ &  $\ddots$ &  $\vdots$ &  $\dots$ &  $\ddots$ \\
\end{tabular}}
\end{center}
In this case, the $q$-th row is called the \emph{first nontrivial strand} when 
$$
\kappa_{1,q}(S(X), V)=h^0(\nP^r, \mathscr{I}_{X/\nP^r}(q+1)) \neq 0.
$$
Following \cite{AHK}, we say that $X \subseteq \nP^r$ satisfies \emph{property $\ND(q)$} for an integer $q \geq 1$ if
$$
H^0(\Lambda, \mathscr{I}_{X \cap \Lambda / \Lambda}(q))=0
$$
for every general linear subspace $\Lambda \subseteq \nP^r$ with $\dim \Lambda \geq \codim X$. This property is also considered in \cite[pp.302--303]{NP}. If $X \subseteq \nP^r$ satisfies property $\ND(q)$, then (\ref{eq:h^0(I(q))}) holds. Conversely, even if (\ref{eq:h^0(I(q))}) holds, property $\ND(q)$ may not be satisfied. For instance, if $X \subseteq \nP^4$ is the image of a general projection of the second Veronese surface in $\nP^5$, then $h^0(\nP^4, \mathscr{I}_{X/\nP^4}(2))=0$ but $X \subseteq \nP^4$ does not satisfy property $\ND(2)$. However, when $X \subseteq \nP^r$ is arithmetically Cohen--Macaulay, property $\ND(q)$ is equivalent to (\ref{eq:h^0(I(q))}). 

\medskip

The aim of this paper is to give a simpler alternative proof of the following theorem on upper bounds for graded Betti numbers of projective schemes in the first nontrivial strand, which was first proven by Ahn--Han--Kwak \cite[Theorem 1.1]{AHK} (see also \cite[Theorems 1.2 and 1.3]{HK} for the case that $X$ is a projective variety and $q=1$).

\begin{theorem}\label{thm:main}
Let $X \subseteq \nP V = \nP^r$ be a non-degenerate projective scheme of codimension $e \geq 1$. For an integer $q \geq 1$, assume that $X \subseteq \nP^r$ satisfies property $\ND(q)$. Then
$$
\kappa_{p,q}(S(X), V) \leq {p+q-1 \choose q} {e+q \choose p+q}~\text{ for all $p \geq 0$}.
$$
In particular, $\kappa_{p,q}(S(X), V)=0$ for $p \geq e+1$.
Furthermore, the following are equivalent:
\begin{enumerate}
 \item $\kappa_{p,q}(S(X), V) = {p+q-1 \choose q} {e+q \choose p+q}$ for all $p$ with $1 \leq p \leq e$.
 \item $\kappa_{p,q}(S(X), V) = {p+q-1 \choose q} {e+q \choose p+q}$ for some $p$ with $1 \leq p \leq e$.
 \item $\deg (X) = {e+q \choose q}$ and $X \subseteq \nP^r$ is arithmetically Cohen--Macaulay.
 \item The graded minimal free resolution of $S(X)$ is of the form
 \begin{align*}
 0 \longleftarrow S(X) \longleftarrow S \longleftarrow S(-q-1)^{\oplus \kappa_{1,q}} \longleftarrow S(-q-2)^{\oplus  \kappa_{2,q}}  \longleftarrow  \\
\cdots \longleftarrow S(-q-e+1)^{\oplus  \kappa_{e-1, q}}\longleftarrow S(-q-e)^{\oplus  \kappa_{e,q}} \longleftarrow 0,
 \end{align*}
 where $S:=\Sym V$ and $\kappa_{p,q}={p+q-1 \choose q} {e+q \choose p+q}$ for all $p$ with $1 \leq p \leq e$.
\end{enumerate}
\end{theorem}

We stress that property $\ND(q)$ is crucial in Theorem \ref{thm:main}. To obtain the upper bounds in the theorem, it is not enough to assume that $H^0(\nP^r, \mathscr{I}_{X/\nP^r}(q))=0$ (i.e., the $q$-th row of the Betti table of $S(X)$ may be the first nontrivial strand). The second row of the Betti table of the image of a general projection of the second Veronese surface to $\nP^4$ is the first nontrivial strand, but every nonzero entry in the second row is strictly bigger than the upper bound in Theorem \ref{thm:main} for $e=2$ and $q=2$ (see Example \ref{ex:projv_2(P^2)} (1)).

\medskip

When $X$ is a variety, property $\ND(1)$ is equivalent to  the non-degeneracy of $X \subseteq \nP^r$. If $X$ is not a variety, then $\ND(1)$ is stronger than the non-degeneracy. For instance, the union of two skew lines in $\nP^3$ is non-degenerate but does not satisfy $\ND(1)$. In the case that $X$ is a variety and $q=1$ as in \cite{HK}, the $\kappa_{p,1}(S(X), V)$ attain the maximum if and only if $\deg(X)=e+1$, i.e., $X \subseteq \nP^r$ is a variety of minimal degree, which is automatically arithmetically Cohen--Macaulay. Varieties of minimal degree were classified by del Pezzo--Bertini (see \cite{EH}). When $X \subseteq \nP V$ is the $q$-secant variety of a variety of minimal degree (the Zariski closure of the union of the linear spans of $q$ points on the variety), it is arithmetically Cohen--Macaulay and $\deg(X) = {e+q \choose q}$ so that the $\kappa_{p,q}(S(X), V)$ attain the maximum (see \cite{CK}). It would be interesting to classify arithmetically Cohen--Macaulay projective varieties of degree ${e+ q \choose q}$. 

\medskip

To prove Theorem \ref{thm:main}, partial elimination ideal theory and generic initial ideal theory were used in \cite{HK} and \cite{AHK}, respectively. But our proof is based on Boij--S\"{o}derberg theory \cite{BS, ES} and is completely different from the previous approach. An advantage of our method is that it allows us to prove interesting results about graded Betti numbers very quickly. As in \cite{Erman}, our paper demonstrates the remarkable power of Boij--S\"{o}derberg theory. Other advantages are that we remove the characteristic zero assumption of \cite{AHK} and that we make the role of property $\ND(q)$ more transparent. In addition, we clarify the characterization when the graded Betti numbers attain maximum.

\medskip

Now, consider the section ring $R(X, L):=\bigoplus_{m \geq 0} H^0(X, L^m)$ of a very ample line bundle $L$ on a smooth projective variety $X$. Then $R(X,L)$ is a finitely generated graded module over $\Sym H^0(X, L)$. Set 
$$
K_{p,q}(X, L):=K_{p,q}(R(X,L), H^0(X, L))~\text{ and }~\kappa_{p,q}(X, L):=\dim K_{p,q}(X, L).
$$
We have the following upper bounds for $\kappa_{p,q}(X, L)$. 

\begin{corollary}
Let $X$ be a smooth projective variety of dimension $n \geq 1$, and $L$ be a very ample line bundle on $X$ giving an embedding $X \subseteq \nP H^0(X, L) = \nP^r$ of codimension $e$. For an integer $q \geq 1$, assume that the natural map $\Sym^i H^0(X, L) \to H^0(X, L^i)$ is an isomorphism for $1 \leq i \leq q$ and $X \subseteq \nP^r$ satisfies property $\ND(q)$. Then
$$
\kappa_{p,q}(X, L) \leq {p+q-1 \choose q} {e+q \choose p+q}~\text{ for all $p \geq 0$}.
$$
In particular, $\kappa_{p,q}(X,L)=0$ for $p \geq e+1$.
Furthermore, the following are equivalent:
\begin{enumerate}
 \item $\kappa_{p,q}(X, L) = {p+q-1 \choose q} {e+q \choose p+q}$ for all $p$ with $1 \leq p \leq e$.
 \item $\kappa_{p,q}(X, L) = {p+q-1 \choose q} {e+q \choose p+q}$ for some $p$ with $1 \leq p \leq e$.
 \item $\deg (X) = {e+q \choose q}$ and $X \subseteq \nP^r$ is arithmetically Cohen--Macaulay.
 \item The graded minimal free resolution of $R(X,L)$ is of the form in Theorem \ref{thm:main} $(4)$.
\end{enumerate}
\end{corollary}

\begin{proof}
The assumption implies that $K_{p,q}(X, L) = K_{p,q}(S(X), H^0(X, L))$ for all $p \geq 0$. Then the corollary immediately follows from Theorem \ref{thm:main}. Here note that if $X \subseteq \nP^r$ is arithmetically Cohen--Macaulay, then $R(X,L)=S(X)$.
\end{proof}

This paper is organized as follows. We begin with reviewing Boij--S\"{o}derberg theory in Section \ref{sec:bs}, where we show a key technical lemma (Lemma \ref{lem:bs-bound}) whose proof is elementary. Section \ref{sec:proof} is devoted to the proof of Theorem \ref{thm:main}. Finally, in Section \ref{sec:complements}, we discuss some complementary results. 

\medskip

\noindent \textbf{Acknowledgments.} We are grateful to Sijong Kwak for interesting discussions and encouragement. We would like to thank the referee for carefully reading the manuscript and for helpful suggestions, especially correcting the proof of Lemma \ref{lem:X_1aCM}.

\section{Boij--S\"{o}derberg Theory}\label{sec:bs}

In this section, we review Boij--S\"{o}derberg theory, which tells us that the Betti table of a finitely generated graded module is a positive rational sum of Betti tables of pure modules. This was established in \cite{ES} for Cohen--Macaulay modules and extended to non-Cohen--Macaulay cases in \cite{BS}. We refer to \cite[Lecture 2]{EL} for a brief summary. 

\medskip

We fix an $(r+1)$-dimensional vector space $V$ over $\kk$ and consider the polynomial ring $S:=\Sym V =\kk[x_0, \ldots, x_r]$. A finitely generated graded Cohen--Macaulay $S$-module $P$ is called \emph{pure} if its graded minimal free resolution is a pure resolution, which is an exact sequence of the form
$$
0 \longleftarrow P \longleftarrow S(-d_0)^{b_0} \longleftarrow S(-d_1)^{b_1} \longleftarrow \cdots \longleftarrow S(-d_{\ell-1})^{b_{\ell-1}} \longleftarrow S(-d_{\ell})^{b_{\ell}} \longleftarrow 0
$$
with $\ell \leq r+1$, i.e., the Betti table of $P$ has a single nonzero entry in each column.
Note that every pure resolution can be associated to a \emph{degree sequence} $(d_0, d_1, \ldots, d_{\ell-1}, d_{\ell})$ of length $\ell$, which is a strictly increasing sequence of $\ell+1$ many integers 
$$
d_0 < d_1 < \cdots < d_{\ell-1} < d_{\ell}.
$$
Eisenbud--Schreyer \cite{ES} proved that there exists a pure module whose graded minimal free resolution is the pure resolution corresponding to any given degree sequence. 

\medskip

Let $d:=(d_0, d_1, \ldots, d_{\ell-1}, d_{\ell})$ be a degree sequence of length $\ell$, and $P_d$ be a pure module whose graded minimal free resolution is a pure resolution corresponding to $d$.
Note that $\kappa_{p,q}(P_d, V) \neq 0$ if and only if $q=d_p-p$. Using Herzog--K\"{u}hl equations (see \cite[Proposition 2.1.4]{EL}), one can determine the nonzero graded Betti numbers of $P_d$ up to scaling:
$$
\kappa_{p,d_p-p}(P_d, V)=\kappa_{0,d_0}(P_d, V) \cdot \prod_{\substack{k \geq 1 \\ k \neq p}} \frac{d_k-d_0}{|d_k - d_p|}~~\text{ for $p \geq 1$}.
$$
We denote by $\pi(d)$ the normalized Betti table of $P_d$ so that the nonzero entry of the $0$-th column is $\kappa_{0,d_0}(d):=1$ and the nonzero entry of the $p$-th column is 
$$
\kappa_{p,d_p-p}(d):=\prod_{\substack{k \geq 1 \\ k \neq p}} \frac{d_k-d_0}{|d_k - d_p|}~~\text{ for $p \geq 1$}.
$$
Here we write $\kappa_{p,q}(d)$ for the $(p,q)$-entry of $\pi(d)$. The multiplicity of the ``virtual'' pure module whose Betti table is the normalized Betti table $\pi(d)$ is 
$$
e(d):=\frac{1}{\ell!}  \prod_{k \geq 1} (d_k-d_0).
$$

\begin{example}\label{ex:deqpure}
Consider the degree sequence $d^{e,q}:=(0,q+1, q+2, \ldots, q+e-1, q+e)$ of length $e$. Then
$$
\kappa_{p,q}(d^{e,q}) = {p+q-1 \choose q} {e+q \choose p+q}~~\text{ for $1 \leq p \leq e$}
~\text{ and }~
e(d^{e,q}) = {e+q \choose q}.
$$
The Betti table $\pi(d^{e,q})$ is the following:
\begin{center}
\texttt{ \begin{tabular}{c|cccccc}
         & $0$ & $1$ & $2$ & $\cdots$ & $e-1$ & $e$ \\ \hline
    $0$    & $1$ & $-$ &  $-$ & $\cdots$ & $-$ &$-$ \\
    $1$   & $-$ & $-$ & $-$ & $\cdots$ & $-$ & $-$ \\
    $\vdots$ & $\vdots$ &  $\vdots$ &  $\vdots$ &  $\ddots$ &  $\vdots$ &  $\vdots$ \\
    $q-1$ & $-$ & $-$ & $-$ & $\cdots$ & $-$ & $-$ \\
    $q$ & $-$ & ${q \choose q}{e+q \choose 1+q}$ & ${q+1 \choose q} {e+q \choose 2+q}$  & $\cdots$ & ${e+q-2 \choose q} {e+q \choose e-1+q}$ & ${e+q-1 \choose q} {e+q \choose e+q}$ 
\end{tabular}}
\end{center}

\smallskip

\noindent This is the Betti table of the $q$-secant variety of a variety of minimal degree. 
\end{example}

Let $M$ be a finitely generated graded $S$-module, and $\pi(M)$ be the Betti table of $M$. Eisenbud--Schreyer \cite{ES} and Boij--S\"{o}derberg \cite{BS} proved that there are positive rational numbers $x_1, \ldots, x_m$ and degree sequences $d^1, \ldots, d^m$ of length at least $r+1-\dim(M)$ such that
$$
\pi(M)=\sum_{i=1}^m x_i \pi(d^i).
$$
This is called the \emph{Boij--S\"{o}derberg decomposition} of $\pi(M)$. If we require that the degree sequences $d^1, \ldots, d^m$  be totally ordered, then the positive rational numbers $x_1, \ldots, x_m$ are uniquely determined. As the results of this paper do not rely on such uniqueness, we will not seriously consider this aspect.
Suppose that $d^1, \ldots, d^{m'}$ have length $r+1-\dim(M)$ and $d^{m'+1}, \ldots, d^m$ have length at least $r+2-\dim(M)$. Then the multiplicity of $M$ is given by
$$
e(M) =  \sum_{i=1}^{m'}  x_i e(d^i).
$$

\begin{example}\label{ex:projv_2(P^2)}
$(1)$ Consider the Betti table of the image of a general projection of the second Veronese surface to $\nP^4$:
\begin{center}
\texttt{ \begin{tabular}{c|ccccc}
         & $0$ & $1$ & $2$ & $3$ & $4$\\ \hline
    $0$    & $1$ & $-$ &  $-$ & $-$ &$-$ \\
    $1$   & $-$ & $-$ & $-$ & $-$ & $-$ \\
    $2$ & $-$ & $7$ & $10$ &  $5$ & $1$ 
\end{tabular}}
\end{center}
The Boij--S\"{o}derberg decomposition is
$$
\frac{2}{3}\pi(d^{2,2}) + \frac{7}{30} \pi(d^{3,2}) + \frac{1}{10} \pi(d^{4,2}).
$$
$(2)$ Consider the Betti table of a non-degenerate union of a plane nodal cubic and a smooth conic in $\nP^4$ (see \cite[Example 5.5]{HK}):
\begin{center}
\texttt{ \begin{tabular}{c|cccc}
         & $0$ & $1$ & $2$ & $3$  \\ \hline
    $0$    & $1$ & $-$ &  $-$ & $-$  \\
    $1$   & $-$ & $5$ & $6$ & $2$  \\
    $2$ & $-$ & $1$ & $2$ &  $1$  
\end{tabular}}
\end{center}
Put $\widetilde{d}^{3,1}:=(0,2,3,5)$ and $d:=(0,2,4,5)$.
The Boij--S\"{o}derberg decomposition is
$$
\frac{2}{3}\pi(d^{3,1}) + \frac{2}{15} \pi(\widetilde{d}^{3,1}) +\frac{1}{10} \pi(d)+ \frac{1}{10} \pi(d^{3,2}).
$$
\end{example}

The following easy lemma plays a key role in the proof of Theorem \ref{thm:main}.

\begin{lemma}\label{lem:bs-bound}
Let $d:=(0, d_1, d_2, \ldots, d_e)$ be a degree sequence of length $e$ with $d_1 \geq q+1$ Then
$$
\kappa_{p,q}(d) \leq \underbrace{{p+q-1 \choose q} {e+q \choose p+q}}_{=\kappa_{p,q}(d^{e,q})}~~\text{ for $1 \leq p \leq e$}~\text{ and }~
e(d) \geq \underbrace{{e+q \choose q}}_{=e(d^{e,q})}.
$$
Furthermore, the following are equivalent:
\begin{enumerate}
\item $\kappa_{p,q}(d) = {p+q-1 \choose q} {e+q \choose p+q}$ for all $p$ with $1 \leq p \leq e$.
\item $\kappa_{p,q}(d) = {p+q-1 \choose q} {e+q \choose p+q}$ for some $p$ with $1 \leq p \leq e$.
\item $e(d) = {e+q \choose q}$.
\item $d=d^{e,q}$.
\end{enumerate}
\end{lemma}

\begin{proof}
See Example \ref{ex:deqpure} for the values of $\kappa_{p,q}(d^{e,q})$ and $e(d^{e,q})$. From this, we also get $(4) \Longrightarrow (1)$.
Since $d_1 \geq q+1$, it follows that $d_k \geq q+k$ for all $k \geq 1$. 
Then
$$
e(d)=\frac{1}{e!} \prod_{k \geq 1} d_k \geq \frac{1}{e!} \prod_{k \geq 1} (q+k) = e(d^{e,q}), 
$$
and the equality holds if and only if $d=d^{e,q}$. In particular, $(3) \Longleftrightarrow (4)$. Now, notice that $\kappa_{p,q}(d) \neq 0$ if and only if $d_p=q+p$. Assume that $d_p=q+p$ for $1 \leq p \leq m$. 
When $1 \leq k < p \leq m$, we have
$$
\frac{d_k}{d_p-d_k} = \frac{q+k}{p-k} = \frac{q+k}{|k-p|}.
$$
When $1 \leq p < k \leq e$ and $p \leq m$, we have $d_p=q+p$ and $d_k=q+k+a$ for some integer $a \geq 0$ so that
$$
\frac{d_k}{d_k-d_p} = \frac{q+k+a}{(k-p)+a} \leq \frac{q+k}{k-p} = \frac{q+k}{|k-p|}
$$
and the equality holds if and only if $d_k=q+k$.
Thus
$$
\kappa_{p,q}(d)=\prod_{\substack{k \geq 1 \\ k \neq p}} \frac{d_k}{|d_k - d_p|} \leq \prod_{\substack{k \geq 1 \\ k \neq p}} \frac{q+k}{|k-p|} = \kappa_{p,q}(d^{e,q})~~\text{ for $1 \leq p \leq m$},
$$
and the equality holds for some $p$ with $1 \leq p \leq m$ if and only if $d=d^{e,q}$. This also proves $(2) \Longleftrightarrow (4)$. We have shown that 
$$
(2) \Longleftrightarrow (3) \Longleftrightarrow (4) \Longrightarrow (1).
$$
The implication $(1) \Longrightarrow (2)$ is trivial.
\end{proof}

\section{Proof of Main Theorem}\label{sec:proof}

This section is devoted to the proof of Theorem \ref{thm:main}. Let $X \subseteq \nP V = \nP^r$ be a non-degenerate projective scheme of dimension $n \geq 0$ and codimension $e \geq 1$.

\begin{lemma}\label{lem:thmforaCM}
Theorem \ref{thm:main} holds when $X \subseteq \nP^r$ is arithmetically Cohen--Macaulay.  
\end{lemma}

\begin{proof}
Consider Boij--S\"{o}derberg decomposition 
$$
\pi(S(X)) = \sum_{i=1}^m x_i \pi(d^i),
$$
where $x_1, \ldots, x_m$ are positive rational numbers with $\sum_{i=1}^m x_i = 1$ and $d^1, \ldots, d^m$ are degree sequences of length $e$. In a Boij--S\"{o}derberg decomposition, the support of each $\pi(d^i)$ with $x_i>0$ lies in the support of $\pi(S(X))$. Therefore $d_0^i = 0$ for all $1\le i \le m$.
Furthermore, as $X \subseteq \nP^r$ satisfies property $\ND(q)$, we have $d_1^i \geq q+1$ for all $1 \leq i \leq m$. Then the assertion follows from Lemma \ref{lem:bs-bound}. Here it is clear that $\kappa_{p,q}(S(X), V) = 0$ for either $p=0$ or $p \geq e+1$. 
\end{proof}

The following is shown in \cite[Lemma 2.19 and Theorem 2.20]{AN} (see also \cite[Remark 3.1]{NP}), but we include the proof for the reader's convenience. For the second assertion, property $\ND(q)$ plays a crucial role.

\begin{lemma}\label{lem:reduction}
Assume that $n \geq 1$. For a general element $t \in V$, let $\overline{X}:=X \cap Z(t)$ and $\overline{V}:=V/\langle t \rangle$ so that $\overline{X}\subseteq \nP \overline{V}= \nP^{r-1}$ is a non-degenerate projective scheme of dimension $n-1$ and codimension $e \geq 1$. 
\begin{enumerate}
\item[$(a)$] Set $I_t:=(I_{X/\nP^r} + (t)) /(t) \subseteq I_{\overline{X}/\nP^{r-1}}$ and $S_t:=\Sym \overline{V}/I_t$. 
Then 
$$
K_{p,q}(S(X), V) = K_{p,q}(S_t, \overline{V})~~\text{ for all $p \geq 0$ and $q \geq 0$}.
$$
\item[$(b)$] Suppose that $X \subseteq \nP^r$ satisfies property $\ND(q)$ for an integer $q \geq 1$. Then there is an injective map $K_{p,q}(S(X), V) \hookrightarrow K_{p,q}(S(\overline{X}), \overline{V})$ for every $p \geq 0$. In particular,
$$
\kappa_{p,q}(S(X), V) \leq \kappa_{p,q}(S(\overline{X}), \overline{V})~~\text{ for all $p \geq 0$}.
$$
\end{enumerate}
\end{lemma}

\begin{proof}
$(a)$ Fixing a splitting $V = \overline{V} \oplus \langle t \rangle$, we have an identification
$$
K_{p,q}(S_t, V) = K_{p,q}(S_t, \overline{V}) \oplus K_{p-1,q}(S_t, \overline{V}).
$$
Note that $t$ is a nonzero divisor on $S(X)$. Thus we have a short exact sequence
$$
0 \longrightarrow S(X)(-1) \xrightarrow{~\cdot t~} S(X) \longrightarrow S_t \longrightarrow 0,
$$
which induces a long exact sequence
$$
\cdots \longrightarrow K_{p,q-1}(S(X), V) \xrightarrow{~\cdot t~} K_{p,q}(S(X), V) \longrightarrow K_{p,q}(S_t, V) \longrightarrow K_{p-1,q}(S(X), V) \longrightarrow \cdots.
$$
Notice that the map $K_{p,q-1}(S(X), V) \to K_{p,q}(S(X), V)$ is always zero (see \cite[Example 5.1.6]{EL}). Thus we get
$$
K_{p,q}(S_t, \overline{V}) \oplus K_{p-1,q}(S_t, \overline{V}) = K_{p,q}(S(X), V) \oplus K_{p-1,q}(S(X), V).
$$
As $K_{0,q}(S(X), V) = K_{0,q}(S_t, \overline{V})$, the assertion follows by induction on $p$.

\medskip

\noindent $(b)$ There is a natural surjective map $S_t \to S(\overline{X})$. Noting that $K_t:=I_{\overline{X}/\nP^{r-1}}/I_t$ is the kernel of this map, we have a short exact sequence
$$
0 \longrightarrow K_t \longrightarrow S_t \longrightarrow S(\overline{X}) \longrightarrow 0
$$
of $\Sym \overline{V}$-modules. This induces a long exact sequence
$$
\cdots \longrightarrow K_{p,q}(K_t, \overline{V}) \longrightarrow K_{p,q}(S_t, \overline{V}) \longrightarrow K_{p,q}(S(\overline{X}), \overline{V}) \longrightarrow K_{p-1,q+1}(K_t, \overline{V}) \longrightarrow \cdots.
$$
Since $X \subseteq \nP^r$ satisfies property $\ND(q)$, it follows that $(K_t)_\ell=0$ for $\ell \leq q$. In particular, $K_{p,q}(K_t, \overline{V}) = 0$. Thus the map $K_{p,q}(S_t, \overline{V}) \to K_{p,q}(S(\overline{X}), \overline{V})$ is injective for every $p \geq 0$. Recall from $(a)$ that $K_{p,q}(S(X), V) = K_{p,q}(S_t, \overline{V})$. 
\end{proof}

Let $\Lambda_i \subseteq \nP^r$ be a general linear subspace of dimension $e+i$ for each $0 \leq i \leq n$ with 
$$
\Lambda_0 \subseteq \Lambda_1 \subseteq \cdots \subseteq \Lambda_{n-1} \subseteq \Lambda_n = \nP^r.
$$
We may write $\Lambda_i = \nP V_i = \nP^{e+i}$ for some quotient space $V_i$ of $V$. Set 
$$
X_i:=X \cap \Lambda_i \subseteq \Lambda_i = \nP^{e+i} = \nP V_i.
$$
Then $\dim X_i = i$. We assume from now on that $X \subseteq \nP^r$ satisfies property $\ND(q)$ for a given integer $q \geq 1$. Then $X_i \subseteq \nP^{e+i}$ also satisfies property $\ND(q)$ for every $0 \leq i \leq n$. 

\medskip

Applying Lemma \ref{lem:reduction} successively, we have
$$
\kappa_{p,q}(S(X), V)\leq \kappa_{p,q}(S(X_{n-1}), V_{n-1}) \leq \cdots \leq \kappa_{p,q}(S(X_1), V_1) \leq \kappa_{p,q}(S(X_0), V_0)
$$
for all $p \geq 0$. It is well-known that the homogeneous coordinate ring of a zero-dimensional subscheme of projective space is Cohen--Macaulay. Thus $X_0 \subseteq \nP^e$ is arithmetically Cohen--Macaulay. By Lemma \ref{lem:thmforaCM}, we obtain
$$
\kappa_{p,q}(S(X), V) \leq \kappa_{p,q}(S(X_0), V_0) \leq {p+q-1 \choose q} {e+q \choose p+q}~~\text{ for all $p \geq 0$}.
$$
It only remains to verify the equivalences in Theorem \ref{thm:main}. The following implications
$$
(3) \Longleftrightarrow (4) \Longrightarrow (1) \Longrightarrow (2)
$$
hold by Lemma \ref{lem:thmforaCM}.  Thus we only need to confirm $(2) \Longrightarrow (3)$. 

\begin{lemma}\label{lem:X_1aCM}
Assume that $n \geq 1$. For a general element $t \in V$, let $\overline{X}:=X \cap Z(t)$ and $\overline{V}:=V/\langle t \rangle$ so that $\overline{X}\subseteq \nP \overline{V}= \nP^{r-1}$ is a non-degenerate projective scheme of dimension $n-1$ and codimension $e \geq 1$. 
Suppose that the following hold:
\begin{enumerate}
    \item[$(a)$] $X \subseteq \nP^{r}$ satisfies property $\ND(q)$ for an integer $q \geq 1$.
    \item[$(b)$] $\kappa_{p_0,q}(S(X), V) = \kappa_{p_0, q}(S(\overline{X}), \overline{V}) \neq 0$ for some $p_0$ with $1 \leq p_0 \leq e$.
    \item[$(c)$] $\kappa_{p,\ell}(S(\overline{X}), \overline{V})=0$ for $0 \leq p \leq p_0$ and $\ell \geq q+1$.
    \item[$(d)$]  $\overline{X} \subseteq \nP^{r-1}$ is arithmetically Cohen--Macaulay.
\end{enumerate}
Then $\kappa_{i,j}(S(X), V) = \kappa_{i,j}(S(\overline{X}), \overline{V})$ for all $i \geq 0$ and $j \geq 0$. In particular,  $X \subseteq \nP^{r}$ is arithmetically Cohen--Macaulay.
\end{lemma}

\begin{proof}
Set $I_t:=(I_{X/\nP^r} + (t)) /(t) \subseteq I_{\overline{X}/\nP^{r-1}}$ and $S_t:=\Sym \overline{V}/I_t$. 
By Lemma \ref{lem:reduction}, $K_{i,j}(S(X), V) = K_{i,j}(S_t, \overline{V})$ for all $i \geq 0$ and $j \geq 0$ and there is an injective map $K_{p,q}(S_t, \overline{V}) \hookrightarrow K_{p,q}(S(\overline{X}), \overline{V})$ for $p \geq 0$. First, we claim that $K_{1,q}(S_t, \overline{V}) = K_{1,q}(S(\overline{X}), \overline{V})$. In view of the condition $(b)$, we only need to consider the case that $K_{p_0,q}(S_t, \overline{V}) = K_{p_0,q}(S(\overline{X}), \overline{V})$ for some $p_0$ with $2 \leq p_0 \leq e$.
We have a short exact sequence
$$
0 \longrightarrow K_t \longrightarrow S_t \longrightarrow S(\overline{X}) \longrightarrow 0
$$
of $\Sym \overline{V}$-modules. Note that $K_{i,j}(K_t, \overline{V})=0$ for $i \geq 0$ and $0 \leq j \leq q$. 
We have a commutative diagram with exact rows
$$
\xymatrixcolsep{0.65in}
\xymatrix{
K_{p_0,q}(S_t, \overline{V}) \ar@{^{(}->}[r] \ar[d] &  K_{p_0,q}(S(\overline{X}), \overline{V}) \ar[r]^-{\varphi_{p_0}} \ar@{^{(}->}[d]^-{\rho_{p_0}} & K_{p_0-1,q+1}(K_t, \overline{V}) \ar[d]  \\
K_{p_0-1,q}(S_t, \overline{V}) \otimes \overline{V} \ar@{^{(}->}[r] &  K_{p_0-1,q}(S(\overline{X}), \overline{V}) \otimes \overline{V} \ar[r]^-{\varphi_{p_0-1} \otimes \operatorname{id}_{\overline{V}}} & K_{p_0-2,q+1}(K_t, \overline{V}) \otimes \overline{V},
}
$$
where the vertical maps are induced from the boundary maps in the graded minimal free resolutions. Here, the middle vertical map is injective by \cite[Proposition 2.11]{AN}. The map $K_{p_0,q} (S_t, \overline{V}) \hookrightarrow$ $K_{p_0,q}(S(\overline{X}),\overline{V})$ is injective, and the dimensions agree by $(b)$, hence it is an isomorphism. Therefore $\varphi_{p_0}$ is zero. Now, consider the commutative diagram with exact rows
$$
\xymatrixcolsep{0.65in}
\xymatrix{
K_{p_0,q}(S_t, \overline{V}) \otimes \overline{V}^{\vee} \ar@{^{(}->}[r] \ar[d] &  K_{p_0,q}(S(\overline{X}), \overline{V}) \otimes \overline{V}^{\vee}  \ar[r]^-{\varphi_{p_0} \otimes \operatorname{id}_{\overline{V}^{\vee}}} \ar[d]^-{\rho_{p_0}^{\dagger}} & K_{p_0-1,q+1}(K_t, \overline{V}) \otimes \overline{V}^{\vee} \ar[d]  \\
K_{p_0-1,q}(S_t, \overline{V}) \ar@{^{(}->}[r] &  K_{p_0-1,q}(S(\overline{X}), \overline{V})  \ar[r]^-{\varphi_{p_0-1}} & K_{p_0-2,q+1}(K_t, \overline{V}) ,
}
$$
where the vertical maps are the adjoints of the vertical maps in the above commutative diagram. To derive a contradiction, suppose that  $\rho_{p_0}^{\dagger}$ is not surjective. Then there exists a nonzero $c \in K_{p_0-1, q}(S(\overline{X}), \overline{V})^{\vee}$ annihilating the image of $\rho_{p_0}^{\dagger}$ so that $(c \otimes \operatorname{id}_{\overline{V}}) \circ \rho_{p_0} = 0$. 
Let $F_{\bullet}$ be the graded minimal free resolution of $S(\overline{X})$ over $\overline{S}:=\Sym \overline{V}$, and $F_{\bullet}^{\vee}:=\operatorname{Hom}_{\overline{S}}(F_{\bullet}, \overline{S})$ be the dual complex. 
As $S(\overline{X})$ is Cohen--Macaulay, we have $\operatorname{Ext}_{\overline{S}}^{p_0-1}(S(\overline{X}), \overline{S})=0$, i.e.,
$$
F_{p_0-2}^{\vee} \xrightarrow{~\delta_{p_0-1}~} F_{p_0-1}^{\vee} \xrightarrow{~\delta_{p_0}~} F_{p_0}^{\vee}
$$
is exact. Since $(c \otimes \operatorname{id}_{\overline{V}}) \circ \rho_{p_0} = 0$, we have $\delta_{p_0}(c \otimes 1) = 0$. Thus $c \otimes 1$ lies in the image of $\delta_{p_0-1}$, but the image of $\delta_{p_0-1}$ is contained in  $\overline{\mathfrak{m}} F_{p_0-1}^{\vee}$ by the minimality of $F_{\bullet}$, where $\overline{\mathfrak{m}}$ is the irrelevant ideal of $\overline{S}$. This gives a contradiction, and hence, $\rho_{p_0}^{\dagger}$ is surjective. As $\varphi_{p_0}$ is zero, the surjectivity of $\rho_{p_0}^{\dagger}$ implies that $\varphi_{p_0-1}$ is zero by the commutativity of the second diagram. 
Thus $K_{p_0-1,q}(S_t, \overline{V}) = K_{p_0-1,q}(S(\overline{X}), \overline{V})$. Repeating the arguments, we finally get $K_{1,q}(S_t, \overline{V}) = K_{1,q}(S(\overline{X}), \overline{V})$ as claimed. Since $K_{0,q+1}(S_t, \overline{V})=0$, it follows that $K_{0,q+1}(K_t, \overline{V})=0$. 

\medskip

Next, for $\ell \geq q+1$, consider the exact sequence
$$
\cdots \longrightarrow K_{1,\ell}(S(\overline{X}), \overline{V}) \longrightarrow K_{0, \ell+1}(K_t, \overline{V}) \longrightarrow K_{0,\ell+1}(S_t, \overline{V}) \longrightarrow \cdots. 
$$
By assumption, $K_{1,\ell}(S(\overline{X}), \overline{V})=0$. Since $K_{0,\ell+1}(S_t, \overline{V})=0$, we have $K_{0,\ell+1}(K_t, \overline{V})=0$. Hence $K_{0,j}(K_t, \overline{V})=0$ for all $j \geq 0$, so $K_t=0$. Therefore,
\[
K_{i,j}(S(X), V) = K_{i,j}(S_t, \overline{V}) = K_{i,j}(S(\overline{X}), \overline{V})~~\text{ for all $i \geq 0$ and $j \geq 0$}. \qedhere
\]
\end{proof}

We go back to the proof of Theorem \ref{thm:main}.
Assume that $(2)$ holds. Recall that $X_0 \subseteq \nP^e$ is arithmetically Cohen--Macaulay. As the implication $(2) \Longrightarrow (3)$ for $X_0$ is known by Lemma \ref{lem:thmforaCM}, we assume $n \geq 1$. We have
$$
\kappa_{p,q}(S(X), V) =\kappa_{p,q}(S(X_0), V_0) = {p+q-1 \choose q} {e+q \choose p+q}
$$
for some $p$ with $1 \leq p \leq e$.
By Lemma \ref{lem:thmforaCM}, we get
$$
\deg(X)=\deg(X_0) =  {e+q \choose q}.
$$
Applying Lemma \ref{lem:X_1aCM} successively, we see that $X \subseteq \nP^r$ is arithmetically Cohen--Macaulay.

\section{Complements}\label{sec:complements}
In this section, we present some additional results. Following \cite{EGHP}, we say that a non-degenerate projective scheme $X \subseteq \nP V = \nP^r$ satisfies \emph{property $N_{d,m}$} if
$$
K_{p,j}(S(X), V) = 0~~\text{ for $0 \leq p \leq m$ and $j \geq d$}.
$$
First, we give a quick proof of the following theorem of Ahn--Han--Kwak \cite[Theorem 1.2 and Corollary 3.4]{AHK}.

\begin{theorem}
Let $X \subseteq \nP V = \nP^r$ be a non-degenerate projective scheme of codimension $e \geq 1$, and $q \geq 1$ be an integer.
\begin{enumerate}
\item If $X \subseteq \nP^r$ satisfies property $\ND(q)$, then $\deg(X) \geq {e+q \choose q}$.
\item If $X \subseteq \nP^r$ satisfies property $N_{q+1,e}$, then $\deg(X) \leq {e+q \choose q}$ and the equality holds if and only if the graded minimal free resolution of $S(X)$ is of the form in Theorem \ref{thm:main} $(4)$.
\end{enumerate}
In particular, if $X \subseteq \nP^r$ satisfies property $\ND(q)$ and $N_{q+1,e}$, then $\deg(X) ={e+q \choose q}$ and $X \subseteq \nP^r$ is arithmetically Cohen--Macaulay.
\end{theorem}

\begin{proof}
$(1)$ Let $\Lambda_0 \subseteq \nP^r$ be a general linear subspace of dimension $e$. Then $X_0:=X \cap \Lambda_0 \subseteq \Lambda_0 = \nP^e$ satisfies property $\ND(q)$ and is arithmetically Cohen--Macaulay. Since every degree sequence appearing in the Boij--S\"{o}derberg decomposition of $S(X_0)$ satisfies the conditions of Lemma \ref{lem:bs-bound}, the assertion follows from Lemma \ref{lem:bs-bound}.

\medskip

\noindent $(2)$ Consider the Boij--S\"{o}derberg decomposition
$$
\pi(S(X)) = \sum_{i=1}^m x_i\pi(d^i),
$$
where $x_1, \ldots, x_m$ are positive rational numbers with $\sum_{i=1}^m x_i = 1$ and $d^1, \ldots, d^{m'}$ are degree sequences of length $e$ and $d^{m'+1}, \ldots, d^m$ are degree sequences of length at least $e+1$. If $1 \leq i \leq m'$, then $d_k^i \leq q+k$ for all $1 \leq k \leq e$ so that
$$
e(d^i) = \frac{1}{e!} \prod_{k \geq 1} d_k^i \leq \frac{1}{e!} \prod_{k \geq 1} (q+k) = e(d^{e,q}) = {e+q \choose q}.
$$
Thus
$$
\deg(X) = \sum_{i=1}^{m'} x_i e(d^i) \leq \sum_{i=1}^{m'} x_i {e+q \choose q} \leq  {e+q \choose q}.
$$
Moreover, if $\deg(X) = {e+q \choose q}$, then  $m=m'=1$ and $d^1 = d^{e,q}$, and hence, $X \subseteq \nP^r$ is arithmetically Cohen--Macaulay. Thus the graded minimal free resolution of $S(X)$ is of the form in Theorem \ref{thm:main} $(4)$. The converse is also immediate from Theorem \ref{thm:main}.
\end{proof}

Finally, we turn to the problem of determining the next-to-maximal values of the graded Betti numbers in the first nontrivial strand. Under some geometric assumption, we have the following answer for the case of $q=1$ (cf. \cite[Theorems 4.1 and 4.3]{HK}).

\begin{theorem}\label{thm:next-to-maximal}
Let $X \subseteq \nP V = \nP^r$ be a projective scheme of dimension $n \geq 1$ and codimension $e \geq 1$, satisfying property $\ND(1)$. Suppose that $\deg(X) \geq e+2$ and $X \cap \Lambda_0 \subseteq \Lambda_0 = \nP^e$ is in linearly general position for a general linear subspace $\Lambda_0 \subseteq \nP^r$ with $\dim \Lambda_0=e$. Then
$$
\kappa_{p,1}(S(X), V) \leq p {e+1 \choose p+1} - {e \choose p-1}~\text{ for all $1 \leq p \leq e-1$}
$$
and $\kappa_{p,1}(S(X), V)=0$ for $p \geq e$.
Furthermore, the following are equivalent:
\begin{enumerate}
 \item $\kappa_{p,1}(S(X), V) = p {e+1 \choose p+1} - {e \choose p-1}$ for all $p$ with $1 \leq p \leq e-1$.
 \item  $\kappa_{p,1}(S(X), V) = p {e+1 \choose p+1} - {e \choose p-1}$ for some $p$ with $1 \leq p \leq e-1$.
 \item $\deg (X) = e+2$ and $X \subseteq \nP^r$ is arithmetically Cohen--Macaulay. 
 \item The graded minimal free resolution of $S(X)$ is of the form
 \begin{align*}
& 0 \longleftarrow S(X) \longleftarrow S \longleftarrow S(-2)^{\oplus \kappa_{1,1}} \longleftarrow S(-3)^{\oplus  \kappa_{2,1}}  \longleftarrow  \\
& \cdots \longleftarrow S(-e+1)^{\oplus  \kappa_{e-2, 1}} \longleftarrow S(-e)^{\oplus  \kappa_{e-1, 1}}\longleftarrow S(-e-2)\longleftarrow 0,
 \end{align*}
 where $S:=\Sym V$ and $\kappa_{p,1}(S(X), V) = p {e+1 \choose p+1} - {e \choose p-1}$ for all $p$ with $1 \leq p \leq e-1$.
\end{enumerate}
\end{theorem}

\begin{proof}
We follow the same strategy as in proving Theorem \ref{thm:main}.
First, consider the degree sequence $\widetilde{d}^{e,1}:=(0,2,3,\ldots, e-1,e,e+2)$ of length $e$. Then
$$
\kappa_{p,1}(\widetilde{d}^{e,1})= p {e+1 \choose p+1} - {e \choose p-1}~~\text{ for $1 \leq p \leq e-1$}~\text{ and }~e(\widetilde{d}^{e,1}) = e+2.
$$
Moreover, the Betti table $\pi(\widetilde{d}^{e,1})$ is the following:
\begin{center}
\texttt{ \begin{tabular}{c|ccccccc}
         & $0$ & $1$ & $2$ & $\cdots$ & $e-2$ & $e-1$ & $e$ \\ \hline
    $0$    & $1$ & $-$ &  $-$ & $\cdots$ & $-$ &$-$ & $-$ \\
    $1$ & $-$ & $ {e+1 \choose 2} - 1$ & $2 {e+1 \choose 3} - e$  & $\cdots$ & $(e-2){e+1 \choose e-1} - {e \choose e-3}$ & $(e-1){e+1 \choose e} - {e \choose e-2}$ & $-$ \\ 
    $2$    & $-$ & $-$ &  $-$ & $\cdots$ & $-$ &$-$ & $1$  \\
\end{tabular}}
\end{center}
As in Lemma \ref{lem:bs-bound}, one can easily check that if $d=(0,d_1,d_2, \ldots, d_{e-1}, d_e)$ is a degree sequence of length $e$ with $d_1 \geq 2$ and $d_e \geq e+2$, then 
$$
\kappa_{p,1}(d) \leq \underbrace{p {e+1 \choose p+1} - {e \choose p-1}}_{=\kappa_{p,1}(\widetilde{d}^{e,1})}~~\text{ for $1 \leq p \leq e-1$}~\text{ and }~e(d) \geq \underbrace{e+2}_{=e(\widetilde{d}^{e,1})}
$$
and one equality holds (or all equalities hold) if and only if $d=\widetilde{d}^{e, 1}$. 
Using Boij--S\"{o}derberg theory, one can prove that the theorem holds when $X \subseteq \nP V = \nP^r$ is an arithmetically Cohen--Macaulay projective scheme of dimension $n \geq 0$ and $\deg(X) \geq e+2$ under the assumption that $K_{e,1}(S(X), V)=0$ and $X$ satisfies property $\ND(1)$. In particular, the implications $(3) \Longleftarrow (4) \Longrightarrow (1) \Longrightarrow (2)$ of the theorem in general hold.

\medskip

Now, choose a general linear subspace $\Lambda_i \subseteq \nP^r$ with $\dim \Lambda_i=e+i$ for each $0 \leq i \leq n$ such that
$$
\Lambda_0 \subseteq \Lambda_1 \subseteq \cdots \subseteq \Lambda_{n-1} \subseteq \Lambda_n = \nP^r.
$$
We write $\Lambda_i = \nP V_i = \nP^{e+i}$ for some quotient space $V_i$ of $V$, and we set
$$
X_i:=X \cap \Lambda_i \subseteq \Lambda_i = \nP^{e+i} = \nP V_i
$$
so that $\dim X_i = i$. Notice that $X_i \subseteq \nP^{e+i}$ satisfies property $\ND(1)$. We claim that $K_{e,1}(S(X_0), V_0)=0$. If $K_{e-1,1}(S(X_0), V_0)=0$, then \cite[Proposition 2.11]{AN} shows that $K_{e,1}(S(X_0), V_0)=0$. Thus suppose that $K_{e-1,1}(S(X_0), V_0) \neq 0$. By \cite[Proposition 3.3]{NP}, $X_0$ is contained in a rational normal curve in $\nP^e$. Here notice that this assertion is trivial when $\deg(X_0) \leq e+2$. Then \cite[Lemma 2.1]{NP} shows that $K_{e,1}(S(X_0), V_0)=0$. Recall that $X_0 \subseteq \nP^e$ is arithmetically Cohen--Macaulay. Applying Lemma \ref{lem:reduction} and the theorem for arithmetically Cohen--Macaulay projective schemes, we have
$$
\kappa_{p,1}(S(X), V) \leq \kappa_{p,1}(S(X_0), V_0) \leq p {e+1 \choose p+1} - {e \choose p-1}~~\text{ for $1 \leq p \leq e-1$}
$$
and $\kappa_{p,1}(S(X), V)=\kappa_{p,1}(S(X_0), V_0)=0$ for $p \geq e$. In particular, $K_{e,1}(S(X),V) =0$, so we get $(3) 
\Longrightarrow (4)$. The remaining implication $(2) \Longrightarrow (3)$ follows from Lemma \ref{lem:X_1aCM} as in the last part of the proof of Theorem \ref{thm:main}.
\end{proof}

Notice that the assumption in Theorem \ref{thm:next-to-maximal} that $X_0 \subseteq \nP^e$ is in linearly general position holds if a general curve section $X_1 \subseteq \nP^{e+1}$ is an integral scheme but not a strange curve by the general position lemma (see \cite[Lemma 1.1]{Rathmann}). In particular, the theorem holds when $X$ is a projective variety and $\operatorname{char}(\kk)=0$. This kind of geometric assumption is essential for the theorem. Without such a condition, a counterexample to the theorem exists (see Example \ref{ex:projv_2(P^2)} $(2)$). In fact, to prove the theorem, we employ geometric arguments involving zero-dimensional schemes, which also served as a main ingredient in Green’s celebrated $K_{p,1}$-theorem \cite{Green1, Green2} (see also \cite[Section 3.3]{AN} and \cite[Section 3]{NP}). On the other hand, for the classification of arithmetically Cohen--Macaulay varieties of almost minimal degree ($\text{degree}=\text{codim}+2$),  we refer to the works of Fujita \cite{Fujita} and Brodmann--Park \cite{BP}.

\begin{remark}
When $X \subseteq \nP V$ is a non-degenerate projective variety and satisfies $\ND(q)$ for $q \geq 2$, we do not know what should be the next-to-maximal value of $\kappa_{p,q}(S(X), V)$. However, if $X \subseteq \nP V$ is a $q$-secant variety (in which case, the $q$-th row of the Betti table is the first nontrivial strand), then not only the maximal value of $\kappa_{p,q}(S(X), V)$ as in Theorem \ref{thm:main} but also the next-to-maximal value of $\kappa_{p,q}(S(X), V)$ is known (see \cite{CK}). When the $\kappa_{p,q}(S(X), V)$ attain the next-to-maximal (e.g., $X$ is the $q$-secant variety of an elliptic normal curve), the Betti table of $S(X)$ is nothing but the Betti table $\pi(\widetilde{d}^{e,q})$ of the degree sequence 
$$
\widetilde{d}^{e,q} := (0, q+1, q+2, \ldots, q+e-2, q+e-1, 2q+e)
$$
of length $e$. We wonder whether the same is true when $X \subseteq \nP V$ is a non-degenerate projective variety and satisfies $\ND(q)$. It would be particularly interesting to know whether $K_{e,q}(S(X), V)=0$ whenever the $\kappa_{p,q}(S(X), V)$ do not attain the maximal value and whether there is an arithmetically Cohen--Macaulay variety whose Betti table is the same as the Betti table $\pi(d)$ of a degree sequence 
$$
d:=(0,q+1, q+2, \ldots, q+e-2, q+e-1, d_e)
$$
of length $e$ with $q+e+1 \leq d_e \leq 2q+e-1$.
\end{remark}


\end{document}